\documentclass[twoside, 12pt]{article}
\usepackage{a4wide, amsthm, amssymb, amscd, mathptmx}

\usepackage{amsfonts}

\usepackage[all]{xy}\CompileMatrices\SelectTips{cm}{12}

\theoremstyle{plain}
\newtheorem{Thm}{\sc Theorem}[section]
\newtheorem{Theorem}[Thm]{\sc Theorem}
\newtheorem*{Theorem*}{\sc Theorem}
\newtheorem{Corollary}[Thm]{\sc Corollary}
\newtheorem*{Lemma*}{\sc Lemma}
\newtheorem{Proposition}[Thm]{\sc Proposition}
\newtheorem*{Proposition*}{\sc Proposition}
\newtheorem{Lemma}[Thm]{\sc Lemma}

\theoremstyle{definition}
\newtheorem{Definition}[Thm]{Definition}

\theoremstyle{remark}
\newtheorem{Remark}[Thm]{Remark}

\newtheorem*{Example*}{Example}
\newtheorem*{Remark*}{Remark}

\newcommand{\D}{{\cal D}}
\newcommand{\E}{{\cal E}}
\newcommand{\F}{{\cal F}}

\renewcommand{\H}{{\cal H}}
\renewcommand{\L}{{\cal L}}

\renewcommand{\O}{{\cal O}}

\newcommand{\HH}{\mathbb{H}}
\newcommand{\NN}{\mathbb{N}}
\newcommand{\PP}{\mathbb{P}}
\newcommand{\QQ}{\mathbb{Q}}
\newcommand{\RR}{\mathbb{R}}
\newcommand{\ZZ}{\mathbb{Z}}

\newcommand{\ch}{\mathop{\rm ch}}

\newcommand{\coker}{\mathop{\rm coker}}

\newcommand{\ext}{{\mathop{{\rm Ext}}}}

\newcommand{\GL}{\mathop{\rm GL}}

\newcommand{\Hilb}{{\mathop{\rm Hilb}}}

\renewcommand{\hom}{{\mathop{\rm Hom}}}
\newcommand{\id}{\mathop{\rm Id}}

\newcommand{\Num}{{\mathop{\rm Num}}}

\newcommand{\rk}{\mathop{\rm rk}}

\newcommand{\td}{\mathop{\rm td}}

\newcommand{\NS}{\mathop{\rm NS}}

\newcommand\arccosh{{\mathop{\rm arccosh}}}
\newcommand{\ccI}{{\cal I}}
\newcommand\mfm{{\mathfrak{m}}}

\begin{document}

\markboth{\rm M.\ Hauzer} {\rm On moduli spaces on Enriques
surfaces}

\title{On moduli spaces of semistable
sheaves on Enriques surfaces}
\author{Marcin Hauzer \footnote{The author tragically died on 27.01.2010.
The paper is a slightly edited version of a part of his
un-submitted PhD thesis written under supervision of Adrian
Langer.}}
\date{\today}

\maketitle

\begin{abstract}
We describe some one-dimensional moduli spaces of rank $2$
Gieseker semistable sheaves on an Enriques surface improving
earlier results of H. Kim. In case of a nodal Enriques surface the
obtained moduli spaces are reducible for general polarizations.

For unnodal Enriques surfaces we show how to reduce the study of
moduli spaces of high even rank Gieseker semistable sheaves to low
ranks. To prove this we use the method of K. Yoshioka who showed
that in the odd rank one can reduce to rank $1$.
\end{abstract}

\section*{Introduction}

An \emph{Enriques surface} is a smooth projective surface $X$
satisfying the following conditions:  the irregularity
$q(X)=h^1(\O_X)$ is equal to $0$ and the canonical line bundle
$\omega_X$ is non-trivial but $\omega_X^{\otimes 2}\simeq \O_X$.
For simplicity we assume that our surfaces are defined over an
algebraically closed field $k$ of characteristic zero (otherwise
we would have to change even the definition of an Enriques
surface).

One of the aims of this note is to study geometry of
one-dimensional moduli spaces of Gieseker semistable sheaves on
Enriques surfaces. We are particularly interested in case of rank
$2$ torsion free sheaves with the first Chern class equal to a
half-pencil of an Enriques surface and with the (degree of) second
Chern class equal to $1$.

Before formulating theorem let us recall that every Enriques
surface is an elliptic fibration over $\PP^1$ with two multiple
fibres $2F_A$ and $2F_B$. $F_A$ and $F_B$ are reduced curves and
they are called \textit{halfpencils}.

\begin{Theorem} \label{main}
There exists an explicit class of polarizations $H$ such that the
moduli space $M_X(2,F_A,1)$ of rank $2$ Gieseker $H$-semistable
sheaves with first Chern class $F_A$ and second Chern class $1$ is
isomorphic to $F_B$.
\end{Theorem}

This result corrects and strengthens the results of Chapter 5 of
Kim's thesis (see \cite[Theorem 5.1]{Kim}).

Let us recall that an Enriques surface is called \emph{unnodal} if
it does not contain any $(-2)$-curves. Kim considered only locally
free sheaves on unnodal Enriques surfaces and claimed that the
corresponding moduli space is non-reduced (which is false). Some
parts of his arguments are also invalid without further
assumptions on the polarization (e.g., in proof of \cite[Theorem
5.1]{Kim} he changes polarization and claims that the bundle
remains stable).

It should be noted that in his later papers H. Kim claimed
somewhat different results. In \cite[Example 1]{Kim2} he claimed a
similar theorem for locally free sheaves and an arbitrary
polarization (this statement is false). In his most recent paper
\cite[II, Example]{Kim3} he claimed the result closest to Theorem
\ref{main}: birationality of the moduli space $M_X(2,F_A,1)$ (for
an arbitrary polarization) with half-pencil $F_B$. In both cases
no proof was provided.

The method of proof od Theorem \ref{main} is quite similar to the
one used by Okonek and Van de Ven \cite{OV} in computation of
Donaldson invariants for Dolgachev surfaces (this result implied
existence of infinitely many homeomorphic surfaces which are not
diffeomorphic). The main new ingredients are a good choice of
polarizations and the method of description of singularities of
moduli spaces (see Subsections \ref{pol} and \ref{singularities}).

\medskip

One of the interests of this theorem stems from the interesting
theorem proven by K. Yoshioka in \cite[Theorem 4.6]{Yo}. Namely,
Yoshioka proved that for a general polarization the moduli space
of semistable sheaves of odd rank and with a primitive Mukai
vector on an unnodal Enriques surface is irreducible. On nodal
Enriques surfaces Theorem \ref{main} provides for a general
polarization an example of a reducible moduli space of semistable
sheaves of even rank and with a primitive Mukai vector.

In even rank for unnodal Enriques surfaces we have the following
theorem:

\begin{Theorem} \label{Main}
Let $X$ be an unnodal Enriques surface and let $r$ be a positive
even integer. Then for a general polarization the number of
irreducible components of the moduli space of rank $r$ Gieseker
semistable sheaves with fixed primitive Mukai vector and with
fixed determinant is the same as the number of irreducible
components of a similar moduli space for rank $2$ or $4$.
\end{Theorem}

This theorem together with Yoshioka's results and Kim's conjecture
in the rank $2$ case suggest that on unnodal Enriques surfaces the
moduli space of Gieseker semistable sheaves with fixed primitive
Mukai vector and determinant should always be irreducible for
general polarization.

In fact we prove a much stronger form of Theorem \ref{Main}
allowing to compare virtual Hodge polynomials of some moduli
spaces (see Theorem \ref{Main2}). Our proof follows the method of
Yoshioka \cite[Section 4]{Yo} but the actual computations become
more complicated than for odd rank. This method of proof also
allows to reprove the main result of \cite{Kim1} (see Theorem
\ref{Kim-thm}).

\medskip

The structure of the paper is as follows. In Section 1 we prove
Theorem \ref{main}. Then in Section 2 we prove a refinement of
Theorem \ref{Main}. At the beginning of each section we describe
the main steps in proofs.

\medskip

In the paper we use without warning the following facts about
Enriques surfaces. If $X$ is an Enriques surface then
$\chi(\O_X)=1$ and the Riemann-Roch theorem for rank $2$ vector
bundle $E$ says that $\chi(E)=2+\frac{1}{2}c_1^2(E)-c_2(E)$. The
canonical divisor of $X$ can be computed as $K_X=F_A-F_B=F_B-F_A$
(see \cite[VII.17]{BHPV}).

\section{One-dimensional moduli spaces of semistable sheaves}

In this section we prove Theorem \ref{main}. The structure of
proof is as follows. First we show how to choose polarizations for
which Theorem \ref{main} holds. Then we show that every $2$
Gieseker $H$-semistable sheaves with first Chern class $F_A$ and
second Chern class $1$ can be obtained as a certain extension.
This is used to prove that we have a set-theoretical bijection
between the corresponding moduli scheme and a half-pencil $F_B$.
The main difficulty is to prove that this is an isomorphism of
schemes. To do so we study singularities of the moduli scheme.
Then we construct some families of sheaves and use them to
construct morphisms from the moduli scheme to the half-pencil and
back providing proof of Theorem \ref{main} (see Theorem
\ref{main2}).

\subsection{Choice of a polarization} \label{pol}

In order to talk about stability of sheaves on a surface $X$ we
have to choose a polarization of $X$. Choosing it smartly we can
exclude existence of strictly semistable sheaves:

\begin{Lemma} \label{polarization}
There exists a polarization $H$ of $X$ such that all rank $2$
Gieseker $H$-semistable sheaves $E$ with $c_1E=F_A$ and $c_2E=1$
are slope $H$-stable and there exists an ample divisor $L_0$ such
that $H=L_0+nF_A$ for some integer $n>F_A.L_0$.
\end{Lemma}

To prove this lemma we need  to recall some results from
\cite[Appendix C to Chapter 4]{HL1}.

Let $X$ be a smooth projective surface. Let $K^+$ denote the
\emph{positive cone} of $X$, i.e., the set
$$\{x\in\Num_{\RR}(X): \, x^2>0\textrm{ and }x.H>0\textrm{ for some
ample divisior }H\}.$$ Let $\H$ denote the set of rays in $K^+$.
This set can be identified with the hyperbolic manifold $\{H\in
K^+||H|=1\}$, where $|H|$ denotes $|H^2|^{\frac{1}{2}}$. We can
define a hyperbolic metric on $\H$ by setting
$$\beta([H],[H'])=\arccosh\left(\frac{H.H'}{|H|.|H'|}\right)$$
for points $[H], [H']\in \H$.

\medskip

\begin{Definition}
Let $r\geq 2$ and $\Delta>0$ be integers. A class $\xi\in\Num (X)$
is \emph{of type} $(r,\Delta)$ if $-\frac{r^2}{4}\Delta
\leq\xi^2<0$. A \emph{wall defined by $\xi$} is the real
1-codimensional submanifold
$$W_{\xi}=\{[H]\in \H|\xi.H=0\}\subset \H.$$
\end{Definition}

\begin{Lemma}
Let us fix positive integers $r\ge 2$ and $\Delta$. Then the set
of walls of type $(r,\Delta)$ is locally finite in $\H$.
\end{Lemma}

\begin{Theorem}
Let $H$ be an ample divisor on $X$ and let $E$ be a slope
$H$-semistable torsion free sheaf of rank $r$ and discriminant
$\Delta$. Let $E'\subset E$ be a subsheaf of rank $r'$, $0<r'<r$,
with $\mu_H(E')=\mu_H(E)$. Then the class
$\xi:=r.c_1(E')-r'.c_1(E)$ satisfies the following conditions:
$$\xi.H=0\quad\textit{and}\quad -\frac{r^2}{4}\Delta \leq
\xi^2\leq 0,$$ and $\xi^2=0$ if and only if $\xi=0$. In
particular, if $c_1\in\Num(X)$ is indivisible, and if $H$ is not
on a wall of type $(r,\Delta)$, then a torsion free sheaf of rank
$r$, with first Chern class $c_1$ and discriminant $\Delta$ is
slope $H$-semistable if and only if it is slope $H$-stable.
\end{Theorem}

Now we can prove Lemma \ref{polarization}:

\medskip

\begin{proof}

An Enriques surface $X$ viewed as an elliptic fibration $X\to \PP
^1$ always has a $2$-section $G$ such that $G.F=2$ for general
fibre (see \cite[Proposition VIII.17.5]{BHPV}). Therefore
$G.F_A=1$ and $c_1=F_A$ is indivisible in $\Num(X)$. To make
$M_X(2,F_A,1)$ parameterize only stable sheaves we should choose a
polarization which is not on a wall of type $(2,-4)$. We also need
to work with a polarization which is close to the ray determined
by $F_A$ in $\Num_{\RR}(X)$ (and which is not ample). More
precisely, the desired polarization should be of the form
$L_0+nF_A$, where $L_0$ is an ample divisor and $n>F_A.L_0$. If we
choose $L_0=L_1$ arbitrary, it may happen that $L_1+nF_A$ is on
the wall of type $(2,-4)$ for every $n$. In this case we fix
$n_0>F_A L_1+1$ and we choose a vector $v\in \Num_{\QQ}(X)$ such
that
\begin{enumerate}
\item $L_1+v$ is an ample divisor,
\item $(L_1+v)+n_0F_A$ is not on a wall of type $(2,-4)$,
\item $v.F_A<1$
\end{enumerate}
Such a vector $v$ exists because the cone of ample divisors is
open and there exists a neighborhood of $L_1+nF_A$ which
intersects only finitely many walls of type $(2,-4)$. Take $m\in
\NN$ such that $L_1':=m(L_1+v)$ belongs to $\Num (X)$. Then $L_1'$
is ample and $L_1'+mn_0F_A=m(L_1+v+n_0F_A)$ does not lie on any
wall of type $(2,-4)$. Moreover, we have
$$L_1'.F_A=m(L_1+v).F_A=m(L_1.F_A+v.F_A)<m(n_0-1+1)=mn_0.$$
Therefore $L_0=L_1'$ and $n=mn_0$ give a polarization $H=L_0+nF_A$
which is not on any wall of type $(2,-4)$ and such that
$n>L_0.F_A$. By Nakai's criterion the divisor $H$ is ample because
$H^2=L^2_0+2nL_0.F_A+F_A^2=L^2_0+2nL_0.F_A>0$ and
$H.D=L_0.D+nF_A.D>0$ since $F_A.D\geq 0$ for every effective
divisor $D$. This finishes the proof.
\end{proof}

\subsection{Presentation of a sheaf as an extension}

From now on we work only with polarizations described in Lemma
\ref{polarization}.

\begin{Lemma} \label{sheaf-to-point}
Let $E$ be a rank $2$ Gieseker $H$-semistable sheaf with the first
Chern class $F_A$ and the second Chern class $1$. Then there
exists a point $x\in X$ such that $E$ sits in a non-split exact
sequence of the form
$$0\rightarrow\O_X\rightarrow E \rightarrow \ccI_x\otimes\O_X(F_A)\rightarrow 0.$$
\end{Lemma}

\begin{proof}
By Lemma \ref{polarization} we know that $E$ is slope $H$-stable.
Since $H$ is fixed we will often omit $H$ when refering to
stability of sheaves. By the Riemann-Roch theorem we have
$$h^0(E)+h^2(E)\ge 2+\frac{1}{2}F_A^2-1=1.$$
By the Serre duality we have $h^2(E)=\dim \hom (E,\omega_X)=
h^0(\omega_X\otimes E^{\vee})$. Moreover, $\omega_X\otimes
E^{\vee}$ is slope stable and locally free. However,
$c_1(\omega_X\otimes E^{\vee})=c_1(E^{\vee})=-F_A$ and
$\mu(\omega_X\otimes E^{\vee})<0$ so $\omega_X\otimes E^{\vee}$
has no global sections and $h^2(E)=0$. Therefore $h^0(E)\ge 1$ and
$E$ fits in an exact sequence
\begin{equation}
\label{first_extension} 0\rightarrow\ccI_{Z_1}(D)\rightarrow E
\rightarrow \ccI_{Z_2}\otimes\O_X(-D+F_A)\rightarrow 0
\end{equation}
for some effective divisor $D$ and zero-dimensional subschemes
$Z_1,Z_2$ of $X$ and such that a section $\O_X\rightarrow E$
factors through $\ccI_{Z_1}(D)\rightarrow E$. By stability of $E$
we have:
\begin{displaymath}
D.H<\frac{1}{2}F_A.H\iff D.(L_0+nF_A)<\frac{1}{2}F_A.(L_0+nF_A)
\iff \\D.L_0+nF_A.D<\frac{1}{2}F_A.L_0
\end{displaymath}
This implies that
$$(F_A.L_0)(F_A.D)<nF_A.D<\frac{1}{2}F_A.L_0 .$$
Therefore $0\le F_A.D<\frac{1}{2}$ and hence $F_A.D=0$.

Now note that computation of the second Chern class from sequence
(\ref{first_extension}) gives
$$D.(F_A-D)+\deg Z_1+\deg Z_2=1.$$
So $D^2=\deg Z_1+\deg Z_2-1\ge -1$. Since the intersection form on
an Enriques surface is even, this implies that $D^2\ge 0$.

By stability of $E$ we also have inequality $(F_A-2D).H>0$. So
$F_A-2D\not =0,K_X$ and since $(F_A-2D)^2=4D^2\ge 0$ by
\cite[Chapter VIII, Proposition 16.1.ii]{BHPV} we know that
$|F_A-2D|\not = \emptyset$. Therefore there exists an effective
divisor $C$ such that $F_A\sim 2D+C$. But $h^0(\O_X(F_A))=1$ (see
the proof of Lemma 17.3 in \cite[Chapter VIII]{BHPV}) and hence we
have equality $F_A=2D+C$. The half-pencil $F_A$ has no multiple
components, so the only possibility is that $D=0$.

Now existence of the morphism $\O_X\rightarrow \ccI_{Z_1}$ shows
that $Z_1$ must be empty. This allows us to compute the length of
$Z_2$:
$$1=c_2(E)=c_2(\ccI_Z\otimes\O_X(-D+F_A))+D.(-D+F_A)=l(Z_2).$$
Splitting of sequence (\ref{first_extension}) would contradict
stability of $E$, so the sequence is non-split.
\end{proof}

Sheaves appearing as extensions of the form from the previous
lemma are characterized by the following lemma:

\begin{Lemma} \label{wyznaczanie}
Let $E$ be a sheaf given by a non-trivial extension
$$0\rightarrow\O_X\rightarrow E \rightarrow \ccI_x\otimes\O_X(F_A)\rightarrow 0$$
for some closed point $x\in X$. Then $x\in F_B$, $E$ is locally
free and it is uniquely determined by $x$.
\end{Lemma}

\begin{proof}
Extensions of $\ccI_x\otimes\O_X(F_A)$ by $\O _X$ are
parameterized by $\ext^1(\ccI_x\otimes\O_X(F_A),\O_X)$. By the
Serre duality this group is dual to $\ext^1(\omega_X,
\ccI_x\otimes\O_X(F_A))=H^1(X,\ccI_x\otimes\O_X(F_A)\otimes
\omega_X)=H^1(X,\ccI_x\otimes\O_X(F_B))$ because $\omega_X\simeq
\O_X(F_B-F_A)$. Consider an exact sequence
$$0\rightarrow\ccI_x\rightarrow \O_X \rightarrow
\O_x\rightarrow 0$$ and tensorize it with $\O_X(F_B)$. By proof of
Lemma 17.3 in \cite[Chapter VIII]{BHPV} we have $h^0(\O_X(F_B))=1$
and $h^1(\O_X(F_B))=0$. Thus we have the following long exact
sequence:
$$0\rightarrow H^0(\ccI_x\otimes\O_X(F_B))\rightarrow H^0(\O_X(F_B))\rightarrow H^0(\O_x(F_B))\rightarrow H^1(\ccI_x\otimes\O_X(F_B)\rightarrow 0.$$
Since $h^0(\O_x(F_B))=1$,  $H^1(\ccI_x\otimes\O_X(F_B))$ has
dimension equal to either $0$ or $1$. But by assumption $E$ comes
from a non-trivial extension, so $h^1(\ccI_x\otimes\O_X(F_B))=1$.
This implies that $h^0(\ccI_x\otimes\O_X(F_B))=1$ and therefore
$x$ lies in the zero set of a nontrivial section of $\O_X(F_B)$,
i.e., $x\in F_B$. Local freeness of $E$ follows from the
Cayley--Bacharach property of a single point $x\in F_B$ with
respect to the linear system $|F_B|=\{F_B\}$.
\end{proof}

\begin{Lemma}
Let $E$ be as in the previous lemma. Then $E$ is slope $H$-stable.
\end{Lemma}

\begin{proof}
To check stability of $E$ it is sufficient to consider subsheaves
of the form  $\O_X(C)\subset E$. We can also assume that this
subsheaf is saturated, i.e., the quotient $E/\O _X(C)$ is torsion
free. If the linear system $|-C|$ is non-empty  then $C.H<0\leq
\frac{1}{2}F_A.H$. So we can assume that $|-C|=\emptyset$. Using
the short exact sequence
$$0\rightarrow\O_X(-C)\rightarrow E(-C)\rightarrow\ccI_x(F_A-C)\rightarrow 0$$
and the fact that $h^0(E(-C))>0$ we get $h^0(\ccI_x(F_A-C))>0$.
Therefore there exists an effective divisor $R\sim F_A-C$ which
passes through $x$. If $R.F_A\geq 1$ then
$$C.H=(L_0+nF_A).(F_A-R)=L_0.F_A-nR .F_A-R.L_0\leq L_0.F_A-n<0\leq
\frac{1}{2}F_A.H.$$ Therefore we can assume that $R.F_A=0$, which
implies that $C.F_A=0$.

Now note that by assumption there exists a zero-dimensional
subscheme $Z$ such that $E$ sits in a short exact sequence of the
form
$$0\rightarrow\O_X (C)\rightarrow E \rightarrow \ccI_Z\otimes \O_X(R)\rightarrow 0.$$
Computing the second Chern class we get $CR+\deg Z=1$. Therefore
$-R^2=CR\le 1$ which implies that $R^2\ge 0$ and $CR\le 0$. But
this implies that $C^2=-CR\ge 0$, so by \cite[Proposition
VIII.16.1]{BHPV} the linear system $|C|$ is non-empty. Therefore
$F_A=C+R$ which contradicts the fact that $x$ lies on $R$.
\end{proof}

Summarizing we have the following corollary:

\begin{Corollary} \label{bijection}
There exists a bijection between closed points  of $M_X(2,1,F_A)$
and $F_B.$
\end{Corollary}

\begin{proof}
The only fact that remained to prove is that for a sheaf $E$, a
point $x\in F_B$ such that we have a non-split exact sequence of
the form
$$0\rightarrow\O_X\rightarrow E \rightarrow \ccI_x\otimes\O_X(F_A)\rightarrow 0$$
is uniquely determined. To prove this note that
$H^0(\ccI_x\otimes\O_X(F_A))=0$, since $x$ does not lie on $F_A$.
Therefore $H^0(E)$ is one-dimensional and $x$ is the zero set of
the unique (up to a multiple by a scalar) non-trivial section of
$E$.
\end{proof}

\subsection{Singularities of $M_X(2,1,F_A)$} \label{singularities}

In order to analyze smoothness of $M_X(2,1,F_A)$ we have to
consider $\ext^2(E,E)=(\hom(E,E\otimes \omega_X))^*$. Since $E$ is
slope stable and of the same slope as $E(K_X)$ every nonzero
homomorphism $s\in\hom(E,E\otimes \omega_X)$ gives rise to an
isomorphism.  Hence $\ext^2(E,E)$ vanishes if and only if $E$ and
$E(K_X)$ are not isomorphic. Both $E$ and $E(K_X)$ represent
points in $M_X(2,F_A,1)$ so we can present them as extensions:
\begin{eqnarray}
\label{ciag1}
0\rightarrow\O_X\rightarrow E\rightarrow\ccI_{x_0}\otimes\O_X(F_A)\rightarrow 0\\
\label{ciag2} 0\rightarrow\O_X\rightarrow
E(K_X)\rightarrow\ccI_{x_1}\otimes\O_X(F_A)\rightarrow 0
\end{eqnarray}
for some uniquely determined $x_0,x_1\in F_B$. In particular, $E$
and $E(K_X)$ are isomorphic if and only if $x_0=x_1$. Now we need
the following lemma:

\begin{Lemma}
Let $\mfm_{x,F_B}$ denote the ideal sheaf of a point $x$ in $F_B$.
Then the sheaves $\mfm_{x_1,F_B}$ and $\mfm_{x_0,F_B}\otimes
\O_{F_B}(F_B)$ are isomorphic as $\O_{F_B}$-modules.
\end{Lemma}

\begin{proof}
Let us consider the following commutative diagram
\begin{center}
$\xymatrix{
&& 0&0&\\
0\ar[r]& {\ker \beta} \ar[r]& \ccI_{x_1}\otimes\O _X(F_A) \ar[r]^{\beta}\ar[u]&\coker \alpha \ar[u]&\\
0\ar[r]&\omega_X\ar[r]\ar[u]_{\gamma}&E(K_X)\ar[r]\ar[u] &\ccI_{x_0}\otimes \O_X(F_B)\ar[r]\ar[u]&0      \\
& 0\ar[r]\ar[u]& \O_X\ar[r]^{=}\ar[u]& \O_X\ar[u]_{\alpha}\ar[r] &0\\
& &0\ar[u]&0\ar[u]&\\
}$
\end{center}
where the middle sequence in this diagram is obtained from
sequence (\ref{ciag1}) by multiplying by $\omega_X$ and using
$K_X=F_B-F_A$. Note that $\gamma$ in this diagram must be an
isomorphism and $\beta$ must be surjective. In particular, we have
a short exact sequence
$$0\to \omega_X{\longrightarrow} \ccI_{x_1}\otimes\O _X(F_A)\longrightarrow \coker \alpha \to 0.$$
By definition we also have an isomorphism $\coker \alpha \simeq
\mfm_{x_0,F_B}\otimes \O_{F_B}(F_B)$. On the other hand the above
exact sequence implies that $\coker \alpha \simeq
\mfm_{x_1,F_B}\otimes \O_{F_B}(F_A)\simeq \mfm_{x_1,F_B}$, which
proves the required assertion.
\end{proof}

\begin{Proposition}
Let $E$ and $E(K)$ be determined by  $x_0,x_1\in F_B$,
respectively.
\begin{enumerate}
\item If $x_0$ is a smooth point of $F_B$ then $x_1$ is also a
smooth point of $F_B$ and it is the zero set of the unique (up to
scalar) section of $\O_{F_B}(F_B+x_0)$. In particular, $x_1\ne
x_0$.
\item If $x_0$ is a singular point of $F_B$ then $x_1=x_0$.
\end{enumerate}
\end{Proposition}

\begin{proof}
Let us first recall that by \cite[Theorem 5.7.5]{CD} the
half-pencils $F_A$ and $F_B$ have only nodal singularities. Let
$C$ be a curve with a nodal singularity at $x$. Then $m_{x,C}$ is
not a line bundle. Otherwise, the maximal ideal of the completion
${\hat{\O} _{C,x}}\simeq k[[a,b]]/(ab)$ of the local ring of $C$
at $x$ would be generated by one element. But this is not
possible.

Let us now assume that $x_0$ is a smooth point of $F_B$. Then
$\mfm_{x_0,F_B}\simeq \O_{F_B}(-x_0)$. Therefore by the above
lemma $\mfm_{x_1,F_B}\simeq \O_{F_B}(F_B-x_0)$ is a line bundle,
which implies that $x_1$ is a smooth point of $F_B$ and
$\O_{F_B}(x_1)\simeq \O_{F_B}(F_B+x_0)$ (we use the fact that
$\O_{F_B}(2F_B)$ is a trivial line bundle). From the short exact
sequence
$$0\to \O_X\to \O_X(F_B) \to \O_{F_B} (F_B)\to 0 $$
we see that $h^0(\O_{F_B}(F_B))=0$ (in particular
$\O_{F_B}(F_B)\simeq \O_{F_B}(x_1-x_0)$ is non-trivial and hence
$x_1\ne x_0$). So from the short exact sequence
$$0\to \O_{F_B}(F_B)\to \O_{F_B}(F_B+x_0)\to \O_{x_0}(F_B+x_0)\simeq \O_{x_0}\to 0$$
we see that $h^0(\O_{F_B}(F_B+x_0))=1$ which proves the first part
of the proposition.

To prove the second part let us assume that $x_0$ is a singular
point of $F_B$. Then $x_1$ is also a singular point of $F_B$,
since $\mfm_{x_1,F_B}$ is not a line bundle. In particular, if
$F_B$ is irreducible then $x_1=x_0$. So we can assume that $F_B$
is reducible. In this case all irreducible components of $F_B$ are
smooth. Let $C$ be an irreducible component of $F_B$ containing
$x_0$. Then we claim that $\mfm_{x_0,F_B}\otimes \O_C \simeq
\O_C(-x_0)\oplus \O_{x_0}$. To prove this note that we have a
canonical surjection $\mfm_{x_0,F_B}\to \mfm_{x_0,C}=\O_C(-x_0)$.
Tensoring it by $\O_C$ we need to prove that the kernel is
isomorphic to the sheaf $\O_{x_0}$. We can do it locally passing
to local completions at the maximal ideal of $\O_{C,x}$. Then the
above map looks like the map
$$(a,b)\otimes _{k[[a,b]]/(ab)} k[[a,b]]/(a) \to (a,b)\cdot k[[a,b]]/(a)\simeq bk[[b]]$$
and the kernel of this map is generated by $a\otimes 1$, which
proves our claim.

The above claim implies that $\mfm_{x_1,F_B}\otimes \O_C$ contains
torsion which is possible only if $x_1$ lies on $C$. But this
implies that $x_1$ lies on the same irreducible components of
$F_B$ as $x_0$ and hence $x_1=x_0$.
\end{proof}

\medskip

The above proposition implies the following corollary:

\begin{Corollary}
\label{E_izo_EK} Let $[E]\in M_X(2,F_A,1)$. Then $E\simeq E(K)$ if
and only if the point $x_0$ associated to $E$ is a singular point
of $F_B$.
\end{Corollary}

\subsection{Family of sheaves} In order to obtain a morphism from
$F_B$ to $M_X(2,1,F_A)$ we have to construct a family of sheaves
$\E$ on $F_B\times X$ such that for every $x\in F_B$ the sheaf
$[\E_{|\{x\}\times X}]\in M_X(2,1,F_A)$. Obviously, we will try to
do in such a way that $\E_{|\{x\}\times X}$ corresponds to the
nontrivial extension of $\ccI_x\otimes\O_X(F_A)$ by $\O_X$.

Let $\Gamma$ denote the graph in the product $F_B\times X$ of the
inclusion $F_B\subset X$ and $\ccI_\Gamma$ be its ideal sheaf. Let
$\pi_i$ denote the projection from $F_B\times X$ on the $i$-th
factor. Observe that, since $\ext^1(\ccI_x\otimes\O_X(F_A),\O_X)$
is one dimensional for every $x\in F_B$, the sheaf
$$\L=\underline{\ext}^1_{\pi_1}(\ccI_\Gamma\otimes \pi^*_2(\O_X(F_A)), \O_{F_B\times X})$$ is a line bundle on $F_B$.
By \cite[p.137]{putinar} there is a spectral sequence
$$H^p(\underline{\ext}^q_{\pi_1}(\ccI_\Gamma\otimes\pi^*_2\O_X(F_A),\pi^*_1(\L^\vee)))\\
\rightarrow
\ext^{p+q}(\ccI_\Gamma\otimes\pi^*_2\O_X(F_A),\pi^*_1(\L^\vee))$$
which gives a long exact sequence:
\begin{eqnarray}
\nonumber 0\rightarrow H^1(\underline{\hom}_{\pi_1}(\ccI_\Gamma\otimes\pi^*_2\O_X(F_A),\pi^*_1(\L^\vee)))\rightarrow \ext^{1}(\ccI_\Gamma\otimes\pi^*_2\O_X(F_A),\pi^*_1(\L^\vee))\rightarrow \\
\nonumber\rightarrow
H^0(\underline{\ext}^1_{\pi_1}(\ccI_\Gamma\otimes\pi^*_2\O_X(F_A),\pi^*_1(\L^\vee)))\rightarrow
H^2(\underline{\hom}_{\pi_1}(\ccI_\Gamma\otimes\pi^*_2\O_X(F_A),\pi^*_1(\L^\vee))).
\end{eqnarray}
For a fixed $x\in F_B$ we have
$\hom(\ccI_x\otimes\O_X(F_A),\O_X)=0$. Hence
$\underline{\hom}_{\pi_1}(\ccI_\Gamma\otimes\pi^*_2\O_X(F_A),\pi^*_1(\L^\vee))=0$.
The isomorphism
\begin{eqnarray}
\nonumber \ext^{1}(\ccI_\Gamma\otimes\pi^*_2\O_X(F_A),\pi^*_1(\L^\vee))\simeq H^0(\underline{\ext}^1_{\pi_1}(\ccI_\Gamma\otimes\pi^*_2\O_X(F_A),\pi^*_1(\L^\vee)))\simeq\\
\nonumber\simeq H^0(\L\otimes\L^{\vee})\simeq H^0(\O_{F_B})
\end{eqnarray}
shows that with $1\in H^0(\O_{F_B})$ we can naturally associate an
extension
$$0\rightarrow\O_{F_B\times X}\rightarrow \E\rightarrow\ccI_\Gamma\otimes\pi^*_2\O_X(F_A)\otimes \pi^*_1(\L)\rightarrow 0$$
on $F_B\times X$, where $\E$ is locally free and after restricting
to $\{x\}\times X$ gives the sheaf associated to $x$. The sheaf
$\E$ can be regarded as a family of sheaves parameterized by
$F_B$. Therefore we get a morphism $F_B\rightarrow M_X(2,F_A,1)$.

\begin{Corollary}
The moduli space $M_X(2,F_A,1)$ is a connected reduced curve.
\end{Corollary}

\begin{proof}
By Lemma \ref{bijection} the morphism $F_B\rightarrow
M_X(2,F_A,1)$ constructed above is bijective on closed points.

If $[E]\in M_X(2,F_A,1)$ corresponds to a smooth point of $F_B$
then by Corollary \ref{E_izo_EK} we have $E\not\simeq E(K)$ and
therefore $\ext^2(E,E)=0$. By \cite[Corollary 5.1.2]{tfs} $[E]$ is
a smooth point and the dimension of $M_X(2,F_A,1)$ at this point
is equal $\dim \ext^1(E,E)=1+c_2(E\otimes
E^{\vee})-4\chi(\O_X)=1$. Therefore $M_X(2,F_A,1)$ is connected
and reduced at every generic point and it has a finite number of
singular points corresponding to singularities of $F_B$.

Note that the expected dimension of the moduli space $
M_X(2,F_A,1)$ at any point $[E]$ is equal to $\dim \ext^1(E,E)-
\dim \ext^2(E,E)=1$. Therefore by \cite[Theorem 4.5.8]{HL1} the
moduli space $M_X(2,F_A,1)$ is a locally complete intersection.
Since $M_X(2,F_A,1)$ is reduced at every generic point it is
reduced everywhere.
\end{proof}

We can also construct a morphism in the opposite direction.

By \cite[Theorem 4.6.5]{HL1} the moduli space $M_X(2,F_A,1)$  is a
fine moduli space. Indeed, the chosen polarisation excludes the
existence of strictly semistable sheaves and if $[E]\in
M_X(2,F_A,1)$ then $\chi(E)=1$ (so we can take $B=\O_X$  in the
above mentioned theorem). Let $\F$ be a universal family on
$M_X(2,F_A,1)\times X$ and let $p_1,p_2$ denote projections on the
first and the second factor, respectively. For every closed point
$[E]\in M_X(2,F_A,1)$ there exists an extension
$$0\rightarrow\O_X\rightarrow E\rightarrow\ccI_{x}\otimes\O_X(F_A)\rightarrow 0$$
for some $x\in F_B$. Hence the long exact sequence of cohomology
gives $H^0(\O_X)\simeq H^0(E)$. Moreover, we have already proved
that $h^2(E)=0$  so equality $\chi(E)=1$ gives us vanishing of
$H^1(E)$. The following theorem shows that ${p_1}_*\F$ is an
invertible sheaf on $M_X(2,F_A,1)$:

\begin{Theorem} \emph{(\cite[Theorem 7.9.9]{ega})}
Let $Y$ be a locally noetherian scheme, $f:X\rightarrow Y$ a
proper morphism, $\F$ a sheaf of $\O_X-$modules flat over $Y$.
Assume that there exists $i_0\in \ZZ$ that
$h^i(f^{-1}(y),\F\otimes_{\O_u}k(y))=0$ for every $i\neq i_0$ and
every $y\in Y$. Then $R^{i_0}f_*\F$ is locally free at $y$ and its
rank is equal to $h^{i_0}(f^{-1}(y),\F\otimes_{\O_u}k(y))$.
\end{Theorem}

Moreover, $\hom({p_1}^*{p_1}_*\F,\F)=\hom({p_1}_*\F,{p_1}_*\F)$ so
we can consider the map
\begin{equation}
\label{morfizm_snopow} {p_1}^*{p_1}_*\F\rightarrow\F
\end{equation}
associated with $\id_{{p_1}_*\F}$. If we look at
(\ref{morfizm_snopow}) on fibres of $p_1$, we recognize the
extension from Lemma (\ref{wyznaczanie}). So the cokernel of
(\ref{morfizm_snopow}) is isomorphic to $\ccI_C\otimes
p_2^*\O_X(F_A)\otimes p_1^*\L'$ for some invertible sheaf $\L'$ on
$M_X(2,F_A,1)$ and a curve $C\subset M_X(2,F_A,1)\times X$. Note
that by restricting $C$ to $[E]\times X$ we get a point $x\in X$
determining $E$. The sheaf $\ccI_C$ gives us a sheaf $\O_C$ which
can be treated as family of zero-dimensional subschemes of $X$
parameterized by $M_X(2,F_A,1)$. This gives a morphism
$M_X(2,F_A,1)\rightarrow \Hilb^1(X)\simeq X$ which factors through
$F_B$.

\begin{Theorem} \label{main2}
Let us fix a polarization $H$ satisfying the conditions from Lemma
\ref{polarization}. Then the moduli space $M_X(2,F_A,1)$ of rank
$2$ Gieseker $H$-semistable sheaves with first Chern class $F_A$
and second Chern class $1$ is isomorphic to $F_B$.
\end{Theorem}

\begin{proof}
We have already constructed  morphisms $M_X(2,F_A,1)\rightarrow
F_B$ and $F_B\rightarrow M_X(2,F_A,1)$ which give identity on
closed points when they are composed. Since both schemes are
reduced these morphisms are isomorphisms.
\end{proof}

\section{Moduli spaces of Gieseker semistable sheaves of even rank}

In this section we prove Theorem \ref{Main}. First we prove some
simple results about lattices. Then we recall some results on the
Mukai lattice for an Enriques surface and we prove some lemmas
concerning this lattice. Finally we use these results and
Yoshioka's method to prove a refinement of Theorem \ref{Main} (see
Theorem \ref{Main2}).

\subsection{Some simple results on lattices} \label{lattices}

Let $L$ be a finitely generated free $\ZZ$-module. An element
$x\in L$ is called \emph{primitive} if the quotient module $L/\ZZ
x$ is torsion free. A \emph{lattice} is a pair consisting of a
finitely generated free $\ZZ$-module and an integral  bilinear (in
our case also symmetric) form $\langle \cdot, \cdot \rangle$.

In the following $-E_8$ denotes lattice $\ZZ ^8$ with canonical
basis $\{e_1,\dots, e_8\}$ whose intersection matrix $(\langle
e_i, e_j\rangle)$ is negative of the Cartan matrix of the root
system $E_8$.

\begin{Lemma}
\label{primitive_element} Let $L$ be a finitely generated free
$\ZZ$-module of rank $\rk L=n>1$. Let $r$ be positive integer and
let $x$ be an element of $L$. Let us set $l=\gcd(r,x)$. Then there
exists $\xi\in L$ such that $\frac{x+r\xi}{l}$ is a primitive
element in $L$. Moreover, if there exists a bilinear form $\langle
\cdot, \cdot \rangle$ such that $(L, \langle  \cdot, \cdot \rangle
)$ is isometric to $-E_8$ then for an arbitrary number $M$ we can
choose $\xi$ such that $2\langle x, \xi \rangle+r \langle \xi ,
\xi \rangle <M$.
\end{Lemma}

\begin{proof}
Let $\{e_1,\ldots,e_n\}$ be a basis for $L$. If $x=0$ then as
$\xi$ we can take an arbitrary primitive element in $L$. If $x\ne
0 $ then we can assume that $\frac{x}{l}=\sum a_ie_i$ with
$a_1\neq 0$. Let $k$ be the product of all prime numbers which
divide $a_1$ but do not divide $a_2$. We claim that for all $b$
such that $\gcd(b,a_1)=1$ the element
$y:=\frac{r}{l}kbe_2+\frac{x}{l}$ is primitive. In our chosen
basis $y$ has coordinates $(a_1,\frac{r}{l}kb+a_2,\ldots, a_n)$.
Let $p$ be a prime divisor of $a_1$. Then either $p|a_m$ for all
$m$ or there exists $m$ such that $p\nmid  a_m$. In the first case
$p\nmid \frac{r}{l}$ because otherwise $\gcd(x,r)>l$. Then $p$
does not divide neither $k$ nor $b$. Therefore
$p\nmid\gcd(a_1,\frac{r}{l}kb+a_2)$ and hence $p\nmid y$.

Now consider the case when there exists $m$ such that $p\nmid
a_m$. If $m\neq 2$ then $p\nmid \gcd(a_1,a_m)$. If $m=2$ then
$p|k$ and $p\nmid\gcd(a_1,\frac{r}{l}kb+a_2)$.

To finish the lemma for lattice $-E_8$ we may assume that $\langle
e_i, e_i\rangle=-2$ and $\langle e_2,e_3\rangle=0$. If $x=0$ then
we take $\xi=pe_2+qe_3$ for prime numbers $p,q\gg 0$. In the other
case it is enough to notice that $2\langle x,kbe_2\rangle+r\langle
kbe_2,kbe_2\rangle$ is a quadratic polynomial in $b$ with negative
leading coefficient so for $b\gg 0$ it less than $M$.
\end{proof}

\medskip

For the convenience of the reader we include a proof of the
following well known lemma.

\begin{Lemma} \label{known}
Let $(L,\langle\,,\,\rangle)$ be a unimodular lattice of rank $n$
and $y\in L$ be a primitive element. Then for every $m\in\ZZ$
there exists $\eta\in L$ such that $\langle \eta,y\rangle=m$.
\end{Lemma}

\begin{proof}
It is enough to prove the lemma for $m=1$. Let $M$ be the
sublattice of $L$ generated by $y$. Then $L/M$ is a free
$\ZZ$-module with basis $[\alpha_1],\ldots,[\alpha_{n-1}]$. Then
$\alpha_1,\ldots,\alpha_{n-1},\alpha_n=y$ is a basis for $L$. The
determinant of the matrix $(\langle \alpha_i,\alpha_j\rangle)$ is
equal $\pm 1$ so
$$\gcd(\langle\alpha_1,\alpha_n\rangle,\ldots,\langle\alpha_i,\alpha_n\rangle,\ldots,\langle\alpha_n,\alpha_n\rangle)=1.$$
Therefore there exist integers $a_i$ such that $\sum
a_i\langle\alpha_i,\alpha_n\rangle=1$ and we can take $\eta=\sum
a_i\alpha_i$.
\end{proof}

\subsection{Mukai's lattice of an Enriques surface}

Let $X$ be a complex Enriques surface and let $K(X)$ be the
Grothendieck group of $X$. Any class in $K(X)$ has well defined
Chern classes. The \emph{Mukai vector} $v(x)$ of a class $x\in
K(X)$ is defined as the following element of $H^{2*}(X,\QQ)$
$$v(x):=\ch(x)\sqrt{{\td}_X}=rk(x)+c_1(x)+\left(\frac{\rk(x)}{2}\varrho_X+{\ch}_2(x)\right)\in H^{2*}(X,\QQ),$$
where $\rho _X$ is the fundamental class of $X$ (i.e., such class
in $H^4(X,\QQ)$ that $\int _X \rho _X=1$). The induced map $v:
K(X)\to H^{2*}(X, \QQ)$ is additive and it factors through the
surjective map $K(X) \to \ZZ \oplus \NS (X) \oplus \ZZ$ given by
$x\to (\rk x, c_1(x), \chi (x))$. This follows from equality $\chi
(x)=\int _X \ch _2 (x)+\rk (x)$ obtained from the Riemann--Roch
theorem. Therefore $v (K (X))=\ZZ \oplus H^2(X,\ZZ)_f\oplus
\frac{1}{2}\ZZ \rho_X\subset H^{2*}(X,\QQ)$, where $H^2(X,\ZZ)_f$
is the torsion free quotient of $H^2(X,\ZZ)$.

On $H^{2*}(X,\QQ)$ we introduce the \emph{Mukai pairing} by
$\langle x,y\rangle:=-\int_X x^\vee \wedge y$. Then the lattice
$(v(K(X)), \langle\cdot,\cdot\rangle)$ is isometric to
$$ \left(\begin{array}{cc}1&0\\0&-1
\end{array}\right)\oplus\left(\begin{array}{cc}0&1\\1&0
\end{array}\right)\oplus -E_8.$$
Note that the Mukai pairing induces on $H^2 (X, \ZZ)_f$
intersection form $(\cdot, \cdot)$.

\begin{Definition}
An element of $v(K(X))$ is called a \emph{Mukai vector}. A Mukai
vector $v$ is \emph{primitive}, if $v$ is primitive as an element
of the lattice $v(K(X))$.
\end{Definition}

\begin{Remark}
\label{uwaga_na_temat_prymitywnosci} A Mukai vector
$v=2+c_1+t\varrho_X$ is not primitive if and only if $c_1$ is
divisible by $2$ and $t$ is odd. Indeed, $v$ can be divisible only
by $2$ and if it is divisible then we have equality
$2+c_1+t\varrho_X=2(1+c_1'+\frac{1}{2}\rho_X
+\frac{1}{2}(c_1')^2-c_2'$) for some $c_1'$ and $c_2'$. This is
equivalent to $c_1=2c_1'$ and $t=1+(c_1')^2-2c_2'$.
\end{Remark}

\medskip

Let us note that for a divisor $D$ we have $v(x\otimes [\O_X
(D)])=v(x)\exp(D)$, where $\exp(D)=1+D+\frac{1}{2}D^2\varrho_X$.
Multiplication by $\exp(D)$ is an isometry of $(v(K(X)),
\langle\,,\,\rangle)$.

In case of Enriques surfaces the torsion free part of the Picard
group is isomorphic to $H^2(X,\ZZ)_f$. We also know that the
lattice $(H^2(X,\ZZ)_f, (\cdot, \cdot))$ is isometric to an
orthogonal direct sum $\HH \perp -E_8$, where $\HH$ is a
hyperbolic plane. The canonical basis of $\HH$ is denoted by
$\{\sigma , f\}$, so we have $\sigma ^2=f^2=0$ and $(\sigma ,
f)=1$.

\medskip

We will also use the following lemma which similarly to Remark
\ref{uwaga_na_temat_prymitywnosci} concerns divisors of $r,c_1$
and $s$ in a primitive Mukai vector.

\begin{Lemma} \label{Lemma_na_temat_gcd}
Let $v=r+c_1-s/2\rho_X$ be a primitive Mukai vector. Then
$\gcd(r,c_1,s)$ equals $1$ or $2$. Moreover:
\begin{itemize}
\item if $\gcd(r,c_1,s)=1$ then either $r$ or $c_1$ is not divisible by $2$,
\item if $\gcd(r,c_1,s)=2$ then $c_2$ must be odd and $r+s\equiv 2\bmod
4$.
\end{itemize}
\end{Lemma}

\begin{proof}
If $\gcd(r,c_1,s)=1$ and $2|\gcd(r,c_1)$ then $s=-r-c_1^2+2c_2$ is
even as well. If a prime number $p>3$ divides $\gcd(r,c_1,s)$ then
$p$ divides $r,c_1$ and $c_2=(r+c_1^2+s)/2$. This is also true for
$p=2$ if we assume that $c_2$ is even. In both these cases $v=pv'$
where $v'$ is a Mukai vector associated to $r'=r/p$, $c_1'=c_1/p$
and $c_2'=c_2/p+(p-1)c_1^2/(2p^2)$. This follows from the
equation:
\begin{eqnarray}
p\left(\frac{r'}{2}+\frac{1}{2}c_1'^2-c_2'\right)&=&p\left(\frac{r}{2p}+\frac{1}{2p^2}c_1^2
-\frac{r+c_1^2+s}{2p}+\frac{p-1}{2p^2}c_1^2\right) ={}\nonumber\\
&=&\frac{r}{2}+\frac{1}{2p}c_1^2-\frac{r}{2}-\frac{1}{2}c_1^2-\frac{s}{2}
+\frac{p-1}{2p}c_1^2=-\frac{s}{2} .\nonumber\end{eqnarray}
Therefore only $2$ can divide $\gcd(r,c_1,s)$ but only if $c_2$ is
odd. If $2|\gcd(r,c_1,s)$ then $r+s\equiv r+c_1^2+s\equiv
2c_2\equiv 2\bmod 4$. It follows that $4\nmid\gcd(r,c_1,s)$.
\end{proof}

\begin{Corollary} \label{Corollary_na_temat_gcd}
Let $v=r+r/2\delta+\xi-s/2\rho_X$ be a primitive Mukai vector,
where $r$ is even and $\delta\in H^2(X,\ZZ)_f$ is primitive. Then
$\gcd(r,\xi,s)$ equals to $1$ or $2$.
\end{Corollary}

\begin{proof}
Let $p>2$ be a prime number such that $p|\gcd(r,\xi,s)$. Then
$p|\gcd(r,r/2\delta+\xi,s)$ which equals $1$ or $2$. Suppose that
$4|\gcd(r,\xi,s)$. Then $\gcd(r,r/2\delta+\xi,s)=2$ and by the
above lemma $r+s\equiv 2\bmod 4$ which leads to a contradiction.
\end{proof}

\medskip

\subsection{Moduli spaces of sheaves of even rank on unnodal Enriques surfaces}

Let $H$ be an ample divisor on $X$. Let $v=r+c_1-(s/2)\varrho_X\in
H^*(X,\QQ)$ be a Mukai vector and let $L$ be a line bundle on $X$
such that $c_1(L)=c_1$. Then by $M_H(v, L)$ we denote the moduli
space of Gieseker $H$-semistable sheaves with Mukai vector $v$ and
with fixed determinant $L$. In the following for a fixed Mukai
vector $v$, $M_H(v,L)$ denotes $M_H(v,L)$ for some fixed line
bundle $L$ satisfying $c_1(L)=c_1(v)$. This will not cause any
problems since there are only two line bundles with the same first
Chern class and they differ  by a torsion line bundle so the
corresponding moduli spaces are isomorphic.

For a complex variety $Y$ the cohomology with compact support has
a natural mixed Hodge structure. This allows us to define the
\emph{virtual Hodge number} $e^{p,q}(Y)=\sum _k (-1)^k h^{p,q}
(H^k_c(Y, \QQ))$ and the \emph{virtual Hodge polynomial} $e(Y)=
\sum _{p,q} e ^{p,q} (Y) x^p y ^q$.

The moduli space $M_H(v, L)$ is constructed as a quotient of a
certain open subset $Q$ of the Quot-scheme by an action of $\GL
(V)$. Then  the rational function $$e(M_H(v, L))=e(Q)/e(\GL (V))$$
is well defined and we call it the \emph{virtual Hodge polynomial}
of $M_H(v,L)$. It is known that for a general polarization $H$ it
does not depend on the choice of $H$ (see \cite[Proposition
4.1]{Yo}).

In \cite{Yo} showed that for an unnodal Enriques surface if
$v=r+c_1-(s/2)\varrho_X\in H^*(X,\QQ)$ is a primitive Mukai vector
such that $r$ is odd  then the virtual Hodge polynomials
$e(M_H(v,L))$ and $e(\Hilb_X^{(\langle v^2\rangle+1)/2})$ are the
same for general $H$. We want to obtain a similar result for even
rank $r$.

The main ingredient of proof of Yoshioka's theorem is the
following proposition:

\begin{Proposition} \emph{(see \cite[Proposition 4.5]{Yo})}
\label{switch_rs} Let $X$ be an unnodal Enriques surface. Assume
that $r,s>0$. If $c_1^2<0$ then for a general polarization $H$ we
have $e(M_H(r+c_1-(s/2)\varrho_X, L ))=e(M_H(s-c_1-(r/2)\varrho_X,
L'))$.
\end{Proposition}

Note that for a Mukai vector $v=r+c_1-(s/2)\varrho_X$ the
condition $(c_1^2)<0$ is equivalent to $\langle v^2\rangle<rs$.

\begin{Theorem} \label{Main2}
Let $X$ be an unnodal Enriques surface and let
$v=r+c_1-(s/2)\varrho_X\in H^*(X,\QQ)$ be a primitive Mukai vector
such that $r$ is even. Then for a general polarization $H$ we have
equality $e(M_H(v,L))=e(M_H(r'+c_1'-(s'/2)\varrho_X,L'))$, where
$r'$ is equal to either $2$ or $4$.
\end{Theorem}

\begin{proof}
To simplify notation we consider only the moduli spaces $M_H(v)$
without fixed determinant. This is sufficient since $M_H(v)$
consists of two disjoint isomorphic moduli spaces of type $M_H(v.
L)$.

We keep notation from the previous subsection: $H^2(X,\ZZ)_f =\HH
\perp -E_8$ and the canonical basis of $\HH$ is denoted by
$\{\sigma , f\}$.

To prove the theorem first we deal with the following special
case: $c_1=\frac{r}{2}b f+\xi$, where $b\in \{ 0,1, -1 \}$ and
$\xi \in -E_8$.  Let us set $l=\gcd(r,\xi )$. By Lemma
\ref{primitive_element} we can find $\xi_1\in -E_8$ such that
$(\xi+r\xi_1)/l$ is primitive and $s-2
(\xi,\xi_1)-r(\xi_1^2)>\langle v^2\rangle$. Therefore replacing
$v$ by $v\exp(\xi_1)$ we may assume that $\xi/l$ is primitive and
$s>\langle v^2\rangle$. Since $v$ is primitive, by Corollary
\ref{Corollary_na_temat_gcd} $\gcd(l,s)$ is equal to either $1$ or
$2$. Since $\xi/ l$  is primitive we have $\gcd (s, \xi
)=\gcd(l,s)$. Now by Proposition \ref{switch_rs} we get
$$e(M_H(r+c_1-(s/2)\varrho_X))=e(M_H(s-c_1-(r/2)\varrho_X)).$$
Note that $s$ is  even as $s+r =-2 \ch _2 (v)$. Replacing
$v=r+c_1-(s/2)\varrho_X$ by $v'=s-c_1-(r/2)\varrho_X$, we may
therefore assume that $r>\langle v^2\rangle$. By the above we may
also assume that $l=\gcd (r, \xi )$ is equal to $1$ or $2$ and
$\xi /l$ is primitive.

Let us set $D=\sigma-\frac{(\eta^2)}{2}f+\eta$, where $\eta$ is
some element of $-E_8$ (note that $(\eta ^2)/2$ is an integer as
$-E_8$ is an even lattice). Then $D^2=0$ and
$$(c_1,D)=\frac{r}{2}b+(\xi, \eta).$$
Let us choose $\eta\in -E_8$ which satisfies the following
conditions:
\begin{itemize}
\item $2(\eta, \xi )=s-2- r b$ if $l=2$ and $4|s-rb-2$ or $l=1$,
\item $2(\eta,\xi )=s-4-rb $ if $l=2$ and $4|s-rb$.
\end{itemize}
Existence of such $\eta$ follows from Lemma \ref{known} because
$\xi /l$ is primitive and the above equalities are equivalent to
$(\eta,\xi /l) =\frac{(s-2-rb)}{2l}$ or $\frac{(s-4-rb)}{2l}$,
respectively.

Then we have
$$v\exp(D)=r+(c_1+rD)+\frac{1}{2}\left( rD^2+2(c_1,D)-s\right)\rho_X
=r+(c_1+rD)-\frac{\varepsilon}{2}\rho_X,$$ where $\varepsilon$ is
equal to either $2$ or $4$. Since $r>\langle v^2\rangle$ we can
use Proposition \ref{switch_rs} once again to obtain
$$e(M_H(v))=e(M_H(v\exp(D)))=e\left(M_H\left(\varepsilon-(c_1+rD)-\frac{r}{2}\rho _X\right)\right).$$
This proves the required assertion in this case.

Analogously exchanging $\sigma$ with $f$ we can deal with the case
$c_1=\frac{r}{2}a \sigma+\xi$, where $a\in \{ 0,1, -1 \}$ and $\xi
\in -E_8$.

\medskip

Now we use induction on $r$ to prove the theorem in the general
case. This part is very similar to the second part of proof of
\cite[Theorem 4.6]{Yo}. Let us write $c_1$ as $d_1\sigma +d_2f
+\xi$ for some $\xi \in -E_8$. Replacing $v$ by $v\exp (k\sigma +l
f)$ we can assume that $-r/2<d_1\le {r/2}$ and $-r/2<d_2\le r/2$.
If $d_1$ is non-zero and $|d_1|<r/2$ then following Yoshioka's
proof we can reduce the assertion to lower rank and use the
induction assumption. Similarly, we deal with the cases when $d_2$
is non-zero and $|d_2|<r/2$. So the only cases that we are left
with are when the pair $(d_1,d_2)$ is equal to $(0,0), (0,r/2),
(r/2,0)$ or $(r/2,r/2)$. But we already proved the theorem in
three of these cases and the only case that is left is
$(d_1,d_2)=(r/2,r/2)$.

In this case we have $c_1=\frac{r}{2}\sigma + \frac{r}{2} f+\xi$
for some $\xi \in -E_8$. To deal with this case we need to
consider another orthogonal decomposition of the lattice $H^2(X,
\ZZ) _f$. Namely, if $\{e_1, \dots , e_8\}$ denotes the canonical
basis of $-E_8$ then we set $\sigma'=\sigma$, $f'=\sigma +f +e_1$,
$e_1'=e_1+2f$ and $e_i'=e_i$ for $i=2,\dots , 8$. Then $\HH '=\ZZ
\sigma ' \oplus \ZZ f'$ is a hyperbolic plane and its orthogonal
complement in $H^2(X, \ZZ)_f$ is isometric to $-E_8$ with
canonical basis $\{e_1', \dots ,e_8' \}$. Let us write $c_1$ in
this new decomposition as $a \sigma'+b f'+\xi'$ for $\xi' \in
-E_8$. Comparing coefficients at $\sigma$ we see that
$a+b=\frac{r}{2}$. This reduces the problem to the already
considered case.
\end{proof}

\medskip

Similar but much simpler considerations lead to another proof of
the following reformulation of Kim's theorem:

\begin{Theorem}\label{Kim-thm}  \emph{(\cite[Theorem]{Kim1})}
Let $v=2+c_1+t\varrho_X\in H^*(X,\QQ)$ be a rank $2$ Mukai vector.
Then there exists a divisor $D$ such that for some $c_1'\in H^2(X,
\ZZ )_f$ we have $v\exp(D)=2+c_1'+0\varrho_X$ or
$v\exp(D)=2+c_1'+\varrho_X$.
\end{Theorem}

\bigskip

{\sc Address:}\\
Institute of Mathematics, Warsaw University, ul.\ Banacha 2,
02-097 Warszawa, Poland,\\


\begin{thebibliography}{SGA1}



\bibitem[BPS]{putinar} G. Banica, M. Putinar, C. Schumacher,
Variation der globalen Ext in Deformationen kompakter komplexer
R\"aume, \emph{ Math. Ann.}  {\bf 250}  (1980), 135--155.

\bibitem[BHPV]{BHPV}
W. Barth, K. Hulek, Ch. Peters,  A. Van de Ven,  Compact complex
surfaces, Second edition. Ergebnisse der Mathematik und ihrer
Grenzgebiete. 3. Folge. A Series of Modern Surveys in Mathematics
{\bf 4}, Springer-Verlag, Berlin, 2004.

\bibitem[CD]{CD} F. Cossec, I. Dolgachev, Enriques surfaces. I. Progress in Mathematics
{\bf 76}, Birkh\"auser Boston, Inc., Boston, MA, 1989.


\bibitem[EGA]{ega} A. Grothendieck,
\'El\'ements de g\'eom\'etrie alg\'ebrique III. \'Etude
cohomologique des faisceaux coh\'erents. II, \emph{Inst. Hautes
\'Etudes Sci. Publ. Math.} {\bf 17} (1963), 5--91.



\bibitem[HL]{HL1} {D. Huybrechts, M. Lehn,} The geometry of moduli
spaces of sheaves, Aspects of Mathematics {\bf 31}, 1997.


\bibitem[Kim1]{Kim}  H. Kim,  Stable vector bundles on Enriques
surfaces, PhD thesis, University of Michigan, 1990.


\bibitem[Kim1]{Kim1} H. Kim, Moduli spaces of bundles mod Picard groups on some elliptic surfaces,
\emph{Bull. Korean Math. Soc} \textbf{35} (1998), 119--125.


\bibitem[Kim2]{Kim2}  H. Kim, Moduli spaces of stable vector bundles
on Enriques surfaces, \emph{Nagoya Math. J.} {\bf 150} (1998),
85--94.

\bibitem[Kim3]{Kim3}  H. Kim, Stable vector bundles of rank two
on Enriques surfaces,  \emph{J. Korean Math.
Soc.} {\bf 43} (2006), 765--782.


\bibitem[La]{tfs} A. Langer, {Lectures on torsion-free sheaves and their
moduli}, Algebraic cycles, sheaves, shtukas, and moduli,  69--103,
Trends Math., Birkh\"auser, Basel, 2008.


\bibitem[OV]{OV} C. Okonek, A. Van de Ven, Stable bundles and differentiable
structures on certain elliptic surfaces, \emph{Invent. Math.} {\bf
86} (1986), 357--370.


\bibitem[Yo]{Yo} K. Yoshioka, Twisted stability and Fourier--Mukai
transform I, \emph{Compositio Math.} {\bf 138} (2003), 261--288.


\end{thebibliography}
\end{document}